\documentclass[11pt,a4paper]{amsart}
\usepackage[T1]{fontenc}
\usepackage[utf8]{inputenc}
\usepackage{lmodern}
\usepackage[a4paper,margin=28mm]{geometry}
\usepackage{amsmath,amssymb,mathtools,amsthm}
\usepackage{microtype}
\usepackage{enumitem}
\usepackage{tikz-cd}
\usepackage[hidelinks]{hyperref}
\usepackage{amsthm}
\newtheorem*{introthmnonum}{Theorem}
\hypersetup{pdftitle={Relative derived delooping levels on intervals of torsion classes},
pdfsubject={Version 7 with unnumbered displayed equations}}
\newtheorem{theorem}{Theorem}[section]
\newtheorem{proposition}[theorem]{Proposition}
\newtheorem{lemma}[theorem]{Lemma}
\newtheorem{corollary}[theorem]{Corollary}

\theoremstyle{definition}
\newtheorem{definition}[theorem]{Definition}
\newtheorem{example}[theorem]{Example}
\theoremstyle{remark}
\newtheorem{remark}[theorem]{Remark}
\DeclareMathOperator{\Fac}{Fac}
\DeclareMathOperator{\Sub}{Sub}
\DeclareMathOperator{\add}{add}
\DeclareMathOperator{\End}{End}
\DeclareMathOperator{\Hom}{Hom}
\DeclareMathOperator{\Ext}{Ext}
\DeclareMathOperator{\rad}{rad}
\DeclareMathOperator{\topm}{top}
\DeclareMathOperator{\tr}{tr}

\DeclareMathOperator{\Findim}{Findim}
\DeclareMathOperator{\findim}{findim}
\DeclareMathOperator{\dell}{dell}
\DeclareMathOperator{\ddell}{ddell}
\DeclareMathOperator{\Proj}{Proj}

\newcommand{\modu}{\operatorname{mod}}
\newcommand{\Modu}{\operatorname{Mod}}
\newcommand{\E}{\mathcal E}
\newcommand{\C}{\mathcal C}
\newcommand{\st}[1]{\underline{#1}}
\newcommand{\op}{\mathrm{op}}

\newcommand{\kk}{K}
\newcommand{\into}{\lhook\joinrel\longrightarrow}
\newcommand{\onto}{\twoheadrightarrow}
\setlist[enumerate]{label=\textup{(\arabic*)},leftmargin=*,itemsep=3pt}
\title[Derived delooping on torsion intervals]{Relative derived delooping levels and torsion intervals}

\author[W. Cao]{Weiqing Cao}
\address{School of Mathematics and Statistics, Jiangsu Normal University,
Xuzhou 221116, Jiangsu, P. R. China}
\email{weiqingcao@jsnu.edu.cn}

\date{}
\makeatletter
\@namedef{subjclassname@2020}{\textup{2020} Mathematics Subject Classification}
\makeatother
\subjclass[2020]{16E05, 16G10}
\keywords{Torsion interval, exact category, relative derived delooping level, support $\tau$-tilting module, finitistic dimension}

\begin{document}
\begin{abstract}
Let $U$  be  a support $\tau$-tilting module over a finite dimensional algebra $A$, and let $\mathcal{V}\subseteq \text{Fac}U$ be any torsion class. The heart
 $\mathcal H=\Fac U\cap \mathcal{V}^\perp$ inherits the exact structure from mod$A$ and admits an Ext-progenerator $P$.  Let $B=\End_A(P)$.  We define the relative derived delooping level on the exact heart
$\mathcal H$ and compare it with the relative delooping level on
$\mathcal H$ and the derived delooping level of $B$. We show that
$\Findim B^{\op}$ is bounded by the relative derived delooping level on the
$\mathcal H$. For a wide heart, the relative derived delooping level on the heart equals
the derived delooping level of $B$. We obtain explicit bounds from representation-finite relative
syzygy strata and characterize global first-order finiteness using an Ext-injective
cogenerator.
\end{abstract}

\maketitle

\section{Introduction}
The finitistic dimension conjecture is one of the central open problems in
the representation theory of Artin algebras. It asks whether the projective dimensions of finitely generated modules of
finite projective dimension admit a common finite bound. The delooping level
introduced by G\'elinas \cite{Gelinas} measures when a syzygy becomes a stable
summand of a higher syzygy and bounds the opposite big finitistic dimension. Guo
and Igusa \cite{GI} introduced derived delooping levels and proved, for every finite dimensional algebra $\Lambda$,
$\Findim\Lambda^{\op}\le\ddell\Lambda\le\dell\Lambda$. Barrios, Lanzilotta and Mata introduced \cite{BLM} global delooping levels and studied their relations with other homological dimensions.
Chen \cite{Chen} showed that finiteness of the ordinary delooping
level is not preserved by derived equivalences. Wei, Wu and
Cao \cite{WeiWuCao} investigated comparisons for derived delooping with a
shift in the delooping degree.

Xu and Zhang \cite{XZ} defined relative delooping levels on $\Fac U$ for a support
$\tau$-tilting module $U$. Their construction uses the Ext-progenerator $U$ and
the tilting equivalence between $\Fac U$ and a torsion free class over $\End_A U$. Their subsequent work \cite{XZ2} proves a comparison with one extra loop and constructs arbitrary large gaps between ordinary and relative levels.
Gao, Liu and Zhang \cite{GLZ} developed derived delooping in exact categories and
applied it to $\Fac U$. Cao, Xu and Zhang \cite{CXZ} extended the relative
delooping level in \cite{XZ} to hearts of torsion intervals with functorially finite
upper endpoint and arbitrary lower torsion class. Motivated by these results,
we extend the derived finite-test construction from
$\operatorname{Fac}U$ to hearts of torsion intervals
with functorially finite endpoints.

Let $A$ be basic and let $T,U\in s\tau\text{-}\operatorname{tilt}A$ satisfy
$\Fac T\subseteq\Fac U$. We abbreviate this relation by $T\le U$. The heart
\[
\mathcal H_{T,U}=\Fac U\cap T^\perp
\]
is extension closed and inherits an exact structure from $\modu A$; cokernels of
arbitrary monomorphisms between its objects need not remain in the heart. By
Enomoto--Sakai \cite[Corollary 4.19]{ES}, $U/\tr_TU$ is an Ext-progenerator of the heart.
Choose a basic module $P=P_{T,U}$ with the same additive closure and put
\[
B=\End_A(P),\qquad F=\Hom_A(P,-).
\]
More generally, if $\mathcal V\subseteq\Fac U$
is any torsion class with torsion radical $t_{\mathcal V}$, then
$U/t_{\mathcal V}(U)$ is an Ext-progenerator of
$\Fac U\cap\mathcal V^\perp$ by
\cite[Lemma~4.17]{ES}. \cite[Theorem 2.7]{CXZ}
proves the same realization for this larger class of intervals.

For $X\in\mathcal H_{T,U}$ and $k\geq1$, write $k\text{-}\ddell_{[T,U]}X$ and
$k\text{-}\dell_{[T,U]}X$ for the relative derived and ordinary levels,
respectively. They are computed using the relative projectives and admissibly
exact witnesses of the heart. The global relative level is
\[
\operatorname{glo}\text{-}k\text{-}\ddell[T,U]
=\sup_{X\in\mathcal H_{T,U}}k\text{-}\ddell_{[T,U]}X.
\]
To control all simple modules of $B$, write $Q_i=F(P_i)$, $P_i\in\mathrm{add}P$, for its indecomposable
projectives and $L_i=\topm Q_i$. For $L_i$ in the image, choose $S_i$ with
$F(S_i)\cong L_i$. For the remaining indices, choose $N_i$ with
$F(N_i)\cong\rad Q_i$. Put
\[
k\text{-}\ddell^{\mathrm{XZ}}[T,U]
=\max\left\{
\begin{array}{l}
\displaystyle\sup_{i\in I}k\text{-}\ddell_{[T,U]}S_i,\\[3pt]
\displaystyle\sup_{j\notin I}\bigl(1+(k+1)\text{-}\ddell_{[T,U]}N_j\bigr)
\end{array}\right\},
\]
where $I$ is the set of simple-image indices (see Definition \ref{def:tests}). Its ordinary counterpart is denoted
$k\text{-}\dell^{\mathrm{XZ}}[T,U]$ and agrees with the $Xu$--$Zhang$ finite-test
invariant of \cite{CXZ}.
The radical correction terms account for the simple
$B$-modules outside $\operatorname{Im}F$.
We compare $k\text{-}\ddell^{\mathrm{XZ}}[T,U]$ with $k\text{-}\dell^{\mathrm{XZ}}[T,U]$ and $k\text{-}\ddell B$, together
with an upper bound for $\operatorname{Findim}B^{\mathrm{op}}$.

\begin{introthmnonum}\label{thm:intro-tests}
\textup{(See Theorems~\ref{thm:canonical} and \ref{thm:wide}.)}
For every $k\ge1$,
\begin{align*}
k\text{-}\ddell B
&\le k\text{-}\ddell^{\mathrm{XZ}}[T,U]
\le k\text{-}\dell^{\mathrm{XZ}}[T,U],\\
k\text{-}\ddell^{\mathrm{XZ}}[T,U]
&\le\operatorname{glo}\text{-}(k+1)\text{-}\ddell[T,U]+1,\\
k\text{-}\ddell^{\mathrm{XZ}}[T,U]
&\le1+(k+1)\text{-}\ddell B.
\end{align*}
Consequently, $\Findim B^{\op}\le\ddell^{\mathrm{XZ}}[T,U]
\le\operatorname{glo}\text{-}2\text{-}\ddell[T,U]+1$.
If the heart is wide, then
\begin{align*}
k\text{-}\ddell_{[T,U]}X&=k\text{-}\ddell_B F(X),\\
k\text{-}\ddell^{\mathrm{XZ}}[T,U]
&=k\text{-}\ddell B,\\
\operatorname{glo}\text{-}k\text{-}\ddell[T,U]
&=\operatorname{glo}\text{-}k\text{-}\ddell B.
\end{align*}
\end{introthmnonum}

The radical contribution has both a shifted order and an additive correction. Example~\ref{ex:sharp} proves that the  additive constant one
in the global upper bound cannot be removed,
even for a connected hereditary algebra with nonzero lower endpoint.

For $s\ge0$, let
\[
\C_s=\add\{\Omega_{T,U}^sX\oplus P'\mid X\in\mathcal H_{T,U},\ P'\in\add P\}.
\]
These are the relative syzygy strata used in \cite{CXZ}. Their finite
stabilization bounds the derived invariants through the ordinary relative bounds.

Suppose that $\C_s$ is representation finite. Let $n$ be the stabilization index
supplied by Lemma~\ref{lem:stabilization}. We show that  for every $k\ge1$ (See Corollary~\ref{prop:finite-stratum}),
\[
k\text{-}\ddell B\le k\text{-}\ddell^{\mathrm{XZ}}[T,U]
\le n+1.
\]
We also construct an Ext-injective cogenerator of
the heart and apply \cite[Theorem 4.14]{GLZ} to characterize finiteness of
$\operatorname{glo}\text{-}\ddell[T,U]$.
Moreover, if $\Ext_A^2(X,Y)=0$ for all $X,Y\in\mathcal H_{T,U}$, then $\operatorname{glo}\text{-}k\text{-}\ddell[T,U]\le1$ (see Proposition \ref{prop:second-ext}).

The paper is organized as follows. Section~\ref{sec:preliminaries} recalls torsion
intervals and the exact realization of their exact hearts.
Section~\ref{sec:derived} defines relative invariants and
establishes their comparisons. Section~\ref{sec:finiteness} treats wide intervals
and the ICE first stratum, constructs Ext-injective cogenerators and gives their
derived-level consequences. Section~\ref{sec:examples} gives some examples.

\section{Preliminaries and interval hearts}\label{sec:preliminaries}

All modules are finitely generated right modules unless $\Modu$ is written.
All full subcategories are closed under isomorphisms, finite direct sums and
direct summands. The symbol $D$ denotes $\Hom_\kk(-,\kk)$. Multiplication in an
endomorphism algebra is composition; precomposition gives $\Hom_A(P,X)$ its
right $\End_A(P)$-module structure. Suprema of empty families of nonnegative
invariants are zero; infima of empty admissible sets are infinity. These
conventions include a zero heart and its zero endomorphism algebra.

\subsection{Torsion intervals}

A torsion class in $\modu A$ is a full subcategory closed under factor modules
and extensions. Write $\operatorname{tors}A$  for their poset under inclusion and
$s\tau\text{-}\operatorname{tilt}A$ for the basic support $\tau$-tilting modules
up to isomorphism. For $T,U\in s\tau\text{-}\operatorname{tilt}A$  with $\Fac T\subseteq\Fac U$, write
$T\le U$ and consider the interval $[\Fac T,\Fac U]$. Its heart is
\[
\mathcal H_{T,U}=\Fac U\cap(\Fac T)^\perp
=\Fac U\cap T^\perp,
\]
where $T^\perp=\{X\mid\Hom_A(T,X)=0\}$.
The second equality follows because objects of $\Fac T$ are quotients of finite
sums of $T$. The heart is extension closed and idempotent complete. Its inherited
conflations are the short exact sequences in $\modu A$ whose three terms belong
to the heart. A wide subcategory is closed under kernels, cokernels and extensions
in $\modu A$.

For modules $M,X$, the trace is
\[
\tr_M(X)=\sum_{f\in\Hom_A(M,X)}\operatorname{Im}f.
\]
For the support $\tau$-tilting module $T$, this is the torsion radical of $\Fac T$
evaluated at $X$.

\begin{theorem}\label{thm:ext-progenerator}\cite{ES}
For $T\le U$, the module $R=U/\tr_TU$ is an Ext-progenerator of
$\mathcal H_{T,U}$. If $P=P_{T,U}$ is basic with $\add P=\add R$, then
\[
\Proj\mathcal H_{T,U}=\add P.
\]
For every $X\in\mathcal H_{T,U}$, there is an exact sequence
\[
0\to K_X\to P_X\to X\to0,
\]
where  $P_X\in\add P$ and $K_X\in\mathcal H_{T,U}$.
\end{theorem}

Fix this basic $P$, and put $B=\End_A(P)$ and
$F=\Hom_A(P,-)$.

\subsection{Stable categories and interval syzygies}

Let $\E$ be a full extension-closed, summand-closed subcategory of $\modu A$ with
Ext-progenerator $P$, equipped with its inherited exact structure. It is a
Hom-finite, Krull--Schmidt exact category with enough projectives. The
projectively stable category is
\[
\st\E=\E/[\add P].
\]
Thus $X$ is a stable retract of $Y$ if maps $X\rightleftarrows Y$ have composite
$1_X$ modulo maps factoring through $\add P$. Relative syzygies $\Omega_\E$ are
defined by projective deflations. Schanuel's lemma makes their stable classes
independent of the choices, and minimal choices are available.

The exact-category arguments below use the same convention in any Hom-finite,
Krull--Schmidt, idempotent-complete exact category with enough projectives.

For the interval heart $\mathcal H_{T,U}$, we denote this syzygy functor by
$\Omega_{T,U}$. Thus $\Omega_{T,U}$ is $\Omega_{\mathcal H_{T,U}}$, computed with
the projectives $\add P_{T,U}$.

\subsection{Exact realization of interval hearts}

\begin{proposition}\label{prop:yoneda}\cite[Proposition 2.4]{CXZ}
The functor $F=\Hom_A(P,-):\mathcal H_{T,U}\to\modu B$ is exact and fully faithful.
It identifies $\add P$ with $\operatorname{proj}B$, induces a fully faithful stable
functor, and preserves syzygies stably. With minimal projective deflations,
$F(\Omega_{T,U}^sX)\cong\Omega_B^sF(X)$ for all $s\ge0$.
\end{proposition}

\begin{proposition}\label{prop:resolving}\cite[Proposition 2.5, Theorem 4.1]{CXZ}
The essential image $\operatorname{Im}F=F(\mathcal H_{T,U})$ is submodule closed
and extension closed in $\modu B$. The functor
$F:\mathcal H_{T,U}\to\operatorname{Im}F$ is an exact equivalence for the
inherited exact structure. Furthermore,
\[
\Proj(\operatorname{Im}F)=\operatorname{proj}B,
\qquad\Omega_B(\modu B)\subseteq\operatorname{Im}F.
\]
In particular, if the heart $\mathcal H_{T,U}$ is wide, then $F$ is an exact equivalence.
\end{proposition}

\begin{theorem}\cite[Theorem 2.7]{CXZ}
Let $U\in\mathrm{s}\tau\text{-}\operatorname{tilt}A$,
let $\mathcal V\subseteq\operatorname{Fac}U$ be any torsion
class, and denote its torsion radical by $t_{\mathcal V}$.
Put
\[
\mathcal H_{\mathcal V,U}
=\operatorname{Fac}U\cap\mathcal V^\perp,
\qquad
R=U/t_{\mathcal V}(U).
\]
Then $R$ is an Ext-progenerator of
$\mathcal H_{\mathcal V,U}$.
If $P$ is a basic representative of $\operatorname{add}R$
and $B=\operatorname{End}_A(P)$, the functor
$\operatorname{Hom}_A(P,-)$ identifies
$\mathcal H_{\mathcal V,U}$ with a submodule-closed,
extension-closed exact subcategory of $\operatorname{mod}B$
containing all projectives.
It preserves minimal projective resolutions.
No functorial finiteness of $\mathcal V$ is required.
\end{theorem}
\begin{remark}
The comparison results of Section~3, the stabilization
criterion, the syzygy-finiteness transfer.
They consequently hold for $\mathcal H_{\mathcal V,U}$
with the same definitions and proofs.
We retain $[T,U]$ in their statements for readability.
The canonical wide realization used for ICE hearts
in Section~4 continues to be stated with two
functorially finite endpoints.
\end{remark}

\section{Relative derived delooping level on intervals }\label{sec:derived}

\subsection{Objectwise and global levels}

A finite sequence
\[
0\longrightarrow C_t\longrightarrow C_{t-1}
\longrightarrow\cdots\longrightarrow C_0
\longrightarrow X\longrightarrow0
\]
in $\E$ is called \emph{admissibly exact} if there are objects
$Z_0=X,Z_1,\ldots,Z_t=C_t$ and conflations
\[
0\longrightarrow Z_{i+1}\longrightarrow C_i
\longrightarrow Z_i\longrightarrow0,
\]
where $ 0\leq i<t$.
For $t=0$, the map $C_0\longrightarrow X$ is required
to be an isomorphism.

\begin{definition}\cite[Definition 2.4]{GLZ}
For $X\in\E$ and  $k\ge1$, define the $k$-delooping level of $X$ by
\[
k\text{-}\dell_\E X
=\inf\{a\ge0\mid\Omega_\E^aX\text{ is a stable retract of }
\Omega_\E^{a+k}Y\text{ for some }Y\in\E\}.
\]
\end{definition}
\begin{definition}\cite[Definition 2.5]{GLZ}
For $X\in\E$ and  $k\ge1$, define the $k$-derived delooping level of $X$ by
\begin{equation*}
\begin{split}
k\text{-}\ddell_{\E}X=\inf\{m\ge0\mid{}&0\le t\le m,\quad
0\to C_t\to\cdots\to C_0\to X\to0\text{ admissibly exact},\\
&(k+i)\text{-}\dell_{\E}C_i\le m-i, ~0\le i\le t\}.
\end{split}
\end{equation*}
\end{definition}
At $k=1$ we omit the prefix. For a finite-dimensional algebra $\Lambda$, its algebra invariant is the maximum over simple modules:
\[
k\text{-}\ddell\Lambda=\sup_{S\text{ simple}}k\text{-}\ddell_{\Lambda}S.
\]
This is different from the supremum over all modules. We use the ordinary analogous convention for $k\text{-}\dell\Lambda$.

We apply the ordinary and derived delooping levels for exact categories \cite{GLZ} to the heart
$\mathcal H_{T,U}$. Following the interval notation of \cite{CXZ}, for $X\in\mathcal H_{T,U}$ and $k\ge1$, we write
\[
\begin{aligned}
k\text{-}\dell_{[T,U]}X&=k\text{-}\dell_{\mathcal H_{T,U}}X,\\
k\text{-}\ddell_{[T,U]}X&=k\text{-}\ddell_{\mathcal H_{T,U}}X.
\end{aligned}
\]
Their global versions are denoted by
\[
\begin{aligned}
\operatorname{glo}\text{-}k\text{-}\dell[T,U]
&=\sup_{X\in\mathcal H_{T,U}}k\text{-}\dell_{[T,U]}X,\\
\operatorname{glo}\text{-}k\text{-}\ddell[T,U]
&=\sup_{X\in\mathcal H_{T,U}}k\text{-}\ddell_{[T,U]}X.
\end{aligned}
\]
For $k=1$, we omit the prefix whenever there is no ambiguity.

For a finite-dimensional algebra $\Lambda$, we likewise write
\[
\begin{aligned}
\operatorname{glo}\text{-}k\text{-}\dell\Lambda
&=\sup_{M\in\modu\Lambda}k\text{-}\dell_{\Lambda}M,\\
\operatorname{glo}\text{-}k\text{-}\ddell\Lambda
&=\sup_{M\in\modu\Lambda}k\text{-}\ddell_{\Lambda}M.
\end{aligned}
\]
This global invariant is distinct from $k\text{-}\ddell\Lambda$, whose supremum is taken over simple modules.

\begin{lemma}\label{lem:basic}\cite[Lemma 2.6 and 2.9, Proposition 2.7]{GLZ}
For every $X\in\E$ and $k\ge1$,
\begin{align*}
k\text{-}\ddell_{\E}X&\le k\text{-}\dell_{\E}X,\\
k\text{-}\ddell_{\E}X&\le(k+1)\text{-}\ddell_{\E}X,\\
k\text{-}\ddell_{\E}X&\le1+(k+1)\text{-}\ddell_{\E}\Omega_\E X.
\end{align*}
\end{lemma}

\begin{lemma}\label{lem:horseshoe}\cite[Lemma 3.2]{GLZ}
If $0\to C_t\to\cdots\to C_0\to X\to0$ is admissibly exact, then for every $s\ge0$, there is an admissibly exact sequence
\[
0\to C_t^{[s]}\to\cdots\to C_0^{[s]}\to\Omega_\E^sX\to0,
\qquad C_i^{[s]}\simeq\Omega_\E^s C_i\text{ in }\st\E.
\]
Consequently,
\begin{equation*}
(k+1)\text{-}\ddell_{\E}\Omega_\E X\le k\text{-}\ddell_{\E}X.
\end{equation*}
\end{lemma}

For each $M\in\operatorname{mod}B$, fix an object
$X_M\in\mathcal H_{T,U}$ such that
\[
F(X_M)\cong\Omega_BM.
\]
Such an object exists by Proposition~\ref{prop:resolving}
and is unique up to isomorphism by the full faithfulness of $F$
in Proposition~\ref{prop:yoneda}.

\begin{lemma}\label{lem:lift-omega}
With the notation above,
 for every $M\in\operatorname{mod}B$
and every integer $k\ge1$, we have
\begin{align*}
k\text{-}\dell_{[T,U]}X_M
&=
k\text{-}\dell_{\operatorname{Im}F}(\Omega_BM)
\le k\text{-}\dell_BM,\\
k\text{-}\ddell_{[T,U]}X_M
&\le k\text{-}\ddell_BM.
\end{align*}
\end{lemma}

\begin{proof}
By Proposition~\ref{prop:resolving}, the subcategory
$\operatorname{Im}F$ contains all projective $B$-modules and is closed
under submodules.

Suppose that $a=k\text{-}\dell_BM<\infty$.
There is a $B$-module $N$ and a stable retraction
\[
\Omega_B^aM
\ \rightleftarrows\
\Omega_B^{a+k}N.
\]
The syzygy functor on the stable category preserves this retraction,
so applying it once gives a stable retraction
\[
\Omega_B^a(\Omega_BM)
\ \rightleftarrows\
\Omega_B^{a+k}(\Omega_BN).
\]
Both $\Omega_BM$ and $\Omega_BN$ belong to $\operatorname{Im}F$.
Moreover, their projective resolutions stay in $\operatorname{Im}F$
and compute the relative syzygies there.
Because $\operatorname{Im}F$ is full and has exactly the same
projective objects as $\modu B$, the maps and their projective
factorizations can also be interpreted inside $\operatorname{Im}F$.
Consequently, the displayed stable retraction gives
\[
k\text{-}\dell_{\operatorname{Im}F}(\Omega_BM)\le a.
\]
The exact equivalence
$F:\mathcal H_{T,U}\longrightarrow\operatorname{Im}F$
identifies the relative projectives and the stable retractions.
Therefore
\[
k\text{-}\dell_{[T,U]}X_M
=
k\text{-}\dell_{\operatorname{Im}F}(\Omega_BM)
\le k\text{-}\dell_BM.
\]

Suppose that $m=k\text{-}\ddell_BM<\infty$.
Choose an exact sequence
\[
0\longrightarrow C_t\longrightarrow\cdots
\longrightarrow C_0\longrightarrow M\longrightarrow0
\]
 and
\[
(k+i)\text{-}\dell_BC_i\le m-i,
\]
where $t\le m$ and $0\le i\le t$.
By  Lemma~\ref{lem:horseshoe},
there is an exact sequence
\[
0\longrightarrow C_t'\longrightarrow\cdots
\longrightarrow C_0'\longrightarrow\Omega_BM\longrightarrow0,
\]
where each $C_i'$ is stably isomorphic to $\Omega_BC_i$.
Each $\Omega_BC_i$ belongs to $\operatorname{Im}F$.
Hence this sequence
is admissibly exact in $\operatorname{Im}F$.

\[
\begin{aligned}
(k+i)\text{-}\dell_{\operatorname{Im}F}C_i'
&=(k+i)\text{-}\dell_{\operatorname{Im}F}(\Omega_BC_i)\\
&\le (k+i)\text{-}\dell_BC_i\\
&\le m-i.
\end{aligned}
\]
Thus the new sequence has the same length $t\le m$ and satisfies all
the required estimates.
It follows that
\[
k\text{-}\ddell_{\operatorname{Im}F}(\Omega_BM)\le m.
\]
Then
\[
k\text{-}\ddell_{[T,U]}X_M
\le m
=k\text{-}\ddell_BM.
\]
\end{proof}

\begin{proposition}\label{prop:object-ordinary}\cite[Proposition 3.7]{CXZ}
For $X\in\mathcal H_{T,U}$ and $k\ge1$,
\[
k\text{-}\dell_BF(X)
\le k\text{-}\dell_{[T,U]}X
\le(k+1)\text{-}\dell_BF(X).
\]
\end{proposition}

\begin{corollary}[Objectwise derived comparison]\label{thm:object}
For $X\in\mathcal H_{T,U}$ and $k\ge1$,
\begin{equation*}
\begin{aligned}
k\text{-}\ddell_{B}F(X)
&\le k\text{-}\ddell_{[T,U]}X\le k\text{-}\dell_{[T,U]}X,\\
k\text{-}\ddell_{[T,U]}X
&\le1+(k+1)\text{-}\ddell_{B}F(X).
\end{aligned}
\end{equation*}
\end{corollary}
\begin{proof}
By Lemma~\ref{lem:basic}, we obtain
\[
k\text{-}\ddell_{[T,U]}X\le k\text{-}\dell_{[T,U]}X.
\]

Suppose that $m=k\text{-}\ddell_{[T,U]}X<\infty$.
Then there is an admissibly exact sequence
\[
0\longrightarrow C_t\longrightarrow\cdots
\longrightarrow C_0\longrightarrow X\longrightarrow0
\]
and
\[
(k+i)\text{-}\dell_{[T,U]}C_i\le m-i,
\]
where $t\le m$ and $0\le i\le t$.
By Proposition~\ref{prop:object-ordinary}, we have
\[
(k+i)\text{-}\dell_B F(C_i)
\le (k+i)\text{-}\dell_{[T,U]}C_i
\le m-i.
\]
Applying $F$ to the sequence above yields an exact sequence in mod$B$
\[
0\longrightarrow F(C_t)\longrightarrow\cdots
\longrightarrow F(C_0)\longrightarrow F(X)\longrightarrow0.
\]
Hence
\[
k\text{-}\ddell_B F(X)\le m
=k\text{-}\ddell_{[T,U]}X.
\]

In either comparison, an infinite right-hand side gives a trivial inequality.

If $(k+1)\text{-}\ddell_B F(X)=\infty$, the inequality is immediate.
Suppose that
\[
m=(k+1)\text{-}\ddell_B F(X)<\infty.
\]
Since $\mathcal H_{T,U}$ has enough relative projectives, there is a
conflation
\[
0\longrightarrow K\longrightarrow P_0
\xrightarrow{p_0}X\longrightarrow0,
\qquad P_0\in\add P_{T,U}, K\in\mathcal H_{T,U}.
\]
Removing redundant projective summands gives a decomposition
$P_0\cong P_X\oplus Q$ with $p_0=(p_X,0)$ and $p_X$ right minimal.
Then
\[
K\cong\ker p_X\oplus Q.
\]
Because $\mathcal H_{T,U}$ is closed under direct summands,
$\ker p_X$ still belongs to $\mathcal H_{T,U}$.
Thus we obtain a minimal relative projective deflation and its conflation
\[
0\longrightarrow\Omega_{T,U}X
\longrightarrow P_X\xrightarrow{p_X}X\longrightarrow0.
\]
By Proposition~\ref{prop:yoneda},
\[
F(\Omega_{T,U}X)\cong\Omega_B F(X).
\]
Applying Lemma~\ref{lem:lift-omega}, we obtain
\[
(k+1)\text{-}\ddell_{[T,U]}(\Omega_{T,U}X)
\le (k+1)\text{-}\ddell_B F(X)=m.
\]
Consequently, there is an admissibly exact sequence
\[
0\longrightarrow C_t\longrightarrow\cdots
\longrightarrow C_0\longrightarrow\Omega_{T,U}X
\longrightarrow0
\]
 and
\[
(k+1+i)\text{-}\dell_{[T,U]}C_i\le m-i,
\]
where $t\le m$ and $0\le i\le t$.
Splicing this sequence with the relative projective conflation above gives
an admissibly exact sequence
\[
0\longrightarrow C_t\longrightarrow\cdots
\longrightarrow C_0\longrightarrow P_X
\longrightarrow X\longrightarrow0.
\]
It is easy to see
\[
k\text{-}\dell_{[T,U]}P_X=0\le m+1.
\]
Each original term $C_i$ is now in position $i+1$, so
\[
\bigl(k+(i+1)\bigr)\text{-}\dell_{[T,U]}C_i
\le m-i=(m+1)-(i+1).
\]
Moreover, the new length is $t+1\le m+1$.
Then
\[
k\text{-}\ddell_{[T,U]}X
\le m+1
=1+(k+1)\text{-}\ddell_B F(X).
\]

\end{proof}

Combining the objectwise comparison with the projective-cover estimate gives the following
comparison of the global derived delooping levels.
\begin{proposition}[Global comparison with a degree shift]\label{thm:global-derived-comparison}
Let $T\le U$ be basic support $\tau$-tilting modules, and put
$B=\End_A(P_{T,U})$. For every integer $k\ge1$,
\begin{align*}
\operatorname{glo}\text{-}k\text{-}\ddell B
&\le1+\operatorname{glo}\text{-}(k+1)\text{-}\ddell[T,U],\\
\operatorname{glo}\text{-}k\text{-}\ddell[T,U]
&\le1+\operatorname{glo}\text{-}(k+1)\text{-}\ddell B.
\end{align*}
\end{proposition}
\begin{proof}
For every $M\in\modu B$, choose $X_M\in\mathcal H_{T,U}$ with
$F(X_M)\cong\Omega_BM$.
Lemma~\ref{lem:basic} and
Corollary~\ref{thm:object} give
\begin{align*}
k\text{-}\ddell_BM
&\le1+(k+1)\text{-}\ddell_B(\Omega_BM)\\
&=1+(k+1)\text{-}\ddell_BF(X_M)\\
&\le1+(k+1)\text{-}\ddell_{[T,U]}X_M\\
&\le1+\operatorname{glo}\text{-}(k+1)\text{-}\ddell[T,U].
\end{align*}
Taking the supremum over $M\in\modu B$ proves the first inequality.
For every $X\in\mathcal H_{T,U}$, Corollary~\ref{thm:object} gives
\[
k\text{-}\ddell_{[T,U]}X
\le1+(k+1)\text{-}\ddell_BF(X)
\le1+\operatorname{glo}\text{-}(k+1)\text{-}\ddell B.
\]
Taking the supremum over $X\in\mathcal H_{T,U}$ proves the second inequality.
\end{proof}

\subsection{A finite-test invariant}
 We now consider the algebra invariant
$k\text{-}\ddell B$, which is computed on simple $B$-modules.
For simple modules in $\operatorname{Im}F$, we use their lifts to the heart.
The remaining simple modules are treated through the lifts of the radicals
of their projective covers.

Fix a decomposition  $P=\bigoplus_{i=1}^nP_i$  into pairwise nonisomorphic indecomposable summands. Put $Q_i=F(P_i)$ and $L_i=\topm Q_i$. The modules $Q_1,\cdots, Q_n$ are the indecomposable projective right $B$-modules, and $L_1,\cdots, L_n$  form a complete list simple right $B$-modules.
Define
\[
I=I(T,U)=\{i\mid L_i\in\operatorname{Im}F\}.
\]
For $i\in I$, choose $S_i\in\mathcal H_{T,U}$ with $F(S_i)\cong L_i$. For $j\notin I$, choose $N_j\in\mathcal H_{T,U}$ with $F(N_j)\cong\rad Q_j$. Existence follows from Proposition~\ref{prop:resolving}; full faithfulness gives uniqueness up to isomorphism.

Using the test objects $S_i$ and $N_j$ above, we first
recall the  $Xu$--$Zhang$ $k$-delooping level of the interval.
\begin{definition}\label{def:ordinary-tests}\cite[Definition 3.9]{CXZ}
For $k\ge1$, define the $Xu$--$Zhang$ $k$-delooping level of the interval by
\begin{equation*}
k\text{-}\dell^{\mathrm{XZ}}[T,U]
=\max\left\{
\begin{array}{l}
\displaystyle\sup_{i\in I}k\text{-}\dell_{[T,U]}S_i,\\[3pt]
\displaystyle\sup_{j\notin I}
\bigl(1+(k+1)\text{-}\dell_{[T,U]}N_j\bigr)
\end{array}\right\}.
\end{equation*}

Now, we define the derived
analogue using the same test objects and replacing
the ordinary relative delooping levels with their
derived counterparts.
\end{definition}
\begin{definition}\label{def:tests}
For $k\ge1$, define the $Xu$--$Zhang$ $k$-derived delooping level of the interval by
\begin{equation*}
k\text{-}\ddell^{\mathrm{XZ}}[T,U]
=\max\left\{
\begin{array}{l}
\displaystyle\sup_{i\in I}k\text{-}\ddell_{[T,U]}S_i,\\[3pt]
\displaystyle\sup_{j\notin I}
\bigl(1+(k+1)\text{-}\ddell_{[T,U]}N_j\bigr)
\end{array}\right\}.
\end{equation*}

\end{definition}
For $T=0$, the objects $S_i$ and $N_j$ are the simple-image
and radical-correction objects of \cite[Definition4.4(2)]{XZ}.
Thus the definition recovers the derived  finite-test invariant in \cite[Definition 4.7]{GLZ}.

The following theorem gives comparison bounds for the derived finite-test invariant of the
interval. These bounds also yield an upper bound for the big finitistic dimension of $B^{\op}$.
\begin{theorem}\label{thm:canonical}
For every $k\ge1$,
\begin{align*}
k\text{-}\ddell B
&\le k\text{-}\ddell^{\mathrm{XZ}}[T,U]
\le k\text{-}\dell^{\mathrm{XZ}}[T,U],
\\
k\text{-}\ddell^{\mathrm{XZ}}[T,U]
&\le
\operatorname{glo}\text{-}(k+1)\text{-}\ddell[T,U]+1,
\\
k\text{-}\ddell^{\mathrm{XZ}}[T,U]
&\le 1+(k+1)\text{-}\ddell B.
\end{align*}
In particular,
\begin{equation*}
\findim B^{\op}
\le \Findim B^{\op}
\le \ddell B
\le \ddell^{\mathrm{XZ}}[T,U].
\end{equation*}
\end{theorem}

\begin{proof}
For $i\in I$, Corollary~\ref{thm:object} gives
\[
k\text{-}\ddell_B L_i
=
k\text{-}\ddell_B F(S_i)
\le
k\text{-}\ddell_{[T,U]}S_i.
\]

Now let $j\notin I$. The projective cover $Q_j\to L_j$ gives
\[
\Omega_B L_j
\cong \rad Q_j
\cong F(N_j).
\]
We claim that
\[
k\text{-}\ddell_B L_j
\le
1+(k+1)\text{-}\ddell_{[T,U]}N_j.
\]
If the right-hand side is infinite, the inequality is immediate.
Otherwise, put
\[
m=(k+1)\text{-}\ddell_{[T,U]}N_j<\infty.
\]
By definition, there is an admissibly exact sequence
\[
0\longrightarrow C_t\longrightarrow\cdots
\longrightarrow C_0\longrightarrow N_j\longrightarrow0
\]
and
\[
(k+1+i)\text{-}\dell_{[T,U]}C_i
\le m-i,
\]
where  $t\le m$ and $0\le i\le t$.
Applying the exact functor $F$ and splicing with
\[
0\longrightarrow F(N_j)\longrightarrow Q_j
\longrightarrow L_j\longrightarrow0
\]
yields an exact sequence
\[
0\longrightarrow F(C_t)\longrightarrow\cdots
\longrightarrow F(C_0)\longrightarrow Q_j
\longrightarrow L_j\longrightarrow0.
\]
In this sequence, $Q_j$ is in position zero and $F(C_i)$
is in position $i+1$. Since $Q_j$ is projective,
\[
k\text{-}\dell_B Q_j=0\le m+1.
\]
Moreover, Proposition~\ref{prop:object-ordinary} gives
\[
\begin{aligned}
\bigl(k+(i+1)\bigr)\text{-}\dell_B F(C_i)
&\le (k+1+i)\text{-}\dell_{[T,U]}C_i\\
&\le m-i\\
&=(m+1)-(i+1).
\end{aligned}
\]
The new length is $t+1\le m+1$. Thus
\[
k\text{-}\ddell_B L_j
\le m+1
=
1+(k+1)\text{-}\ddell_{[T,U]}N_j.
\]
Taking the maximum over all simple $B$-modules proves
\[
k\text{-}\ddell B
\le k\text{-}\ddell^{\mathrm{XZ}}[T,U].
\]
By Lemma \ref{lem:basic},
the second
inequality in the first displayed chain is proved.

For the global bound, Lemma~\ref{lem:basic} gives
\[
\begin{aligned}
k\text{-}\ddell_{[T,U]}S_i
&\le (k+1)\text{-}\ddell_{[T,U]}S_i\\
&\le
\operatorname{glo}\text{-}(k+1)\text{-}\ddell[T,U], ~~~i\in I.
\end{aligned}
\]
For $j\notin I$, we also have
\[
1+(k+1)\text{-}\ddell_{[T,U]}N_j
\le
1+\operatorname{glo}\text{-}(k+1)\text{-}\ddell[T,U].
\]
Taking the maximum of the two families proves
 $$k\text{-}\ddell^{\mathrm{XZ}}[T,U]
\le
\operatorname{glo}\text{-}(k+1)\text{-}\ddell[T,U]+1.$$

Corollary~\ref{thm:object} gives
\[
k\text{-}\ddell_{[T,U]}S_i
\le
1+(k+1)\text{-}\ddell_B L_i,~~~i\in I.
\]
For $j\notin I$, apply Lemma~\ref{lem:lift-omega}
with $M=L_j$ and parameter $k+1$.
Since
\[
F(N_j)\cong\Omega_B L_j,
\]
the full faithfulness of $F$ identifies $N_j$ with the object
$X_{L_j}$ in that lemma. Hence
\[
(k+1)\text{-}\ddell_{[T,U]}N_j
\le
(k+1)\text{-}\ddell_B L_j.
\]
Taking the maximum proves
$k\text{-}\ddell^{\mathrm{XZ}}[T,U]
\le
 1+(k+1)\text{-}\ddell B.$

Finally, combining the inequality in
\cite[Theorem 2.7]{GI}
\[
\Findim B^{\op}\le\ddell B
\]
with the case $k=1$ of the first displayed chain
proves the finitistic-dimension chain.
\end{proof}

\section{Wide intervals and finiteness criteria}\label{sec:finiteness}

\subsection{Wide intervals} The interval [Fac$T$, Fac$U$] is called wide when its heart is wide subcategory of mod$A$, following \cite{ES}. Since the endpoints are functorially finite, the heart is functorially finite, see \cite{AP,ES}.

\begin{theorem}[Wide interval equality]\label{thm:wide}
If $\mathcal H_{T,U}$ is wide in $\modu A$, then
\[
\begin{aligned}
k\text{-}\ddell_{[T,U]}X&=k\text{-}\ddell_BF(X),\\
k\text{-}\ddell^{\mathrm{XZ}}[T,U]&=k\text{-}\ddell B,\\
\operatorname{glo}\text{-}k\text{-}\ddell[T,U]
&=\operatorname{glo}\text{-}k\text{-}\ddell B.
\end{aligned}
\]
\end{theorem}
\begin{proof} The inherited exact structure is now the abelian structure.
Ext-progenerator $P$ is a projective generator, and  $F$ is an exact equivalence. Both $F$ and its quasi-inverse preserve projectives,
syzygies, stable retractions and admissible witnesses, so the objectwise
levels agree. All simple $B$-modules are in the image, hence there are no
radical terms.
\end{proof}

\subsection{A finiteness criterion}

For $s\ge0$, set
\[
\mathcal C_s=\add\{\Omega_{T,U}^sX\oplus P'
\mid X\in\mathcal H_{T,U},\ P'\in\add P\},
\qquad \mathcal{L}_s(T,U)=\operatorname{ind}\mathcal C_s.
\]
In particular, $\mathcal C_0=\mathcal H_{T,U}$. The sets of indecomposable
classes descend, and every stratum contains $\add P$.

\begin{lemma}\label{lem:stabilization}\cite[Lemma 4.2]{CXZ}
Suppose $L_s(T,U)$ is finite for some $s\geq 0$ and put
\[
n=\min\{r\ge s\mid \mathcal{L}_r(T,U)=\mathcal{L}_{r+1}(T,U)\}.
\]
Then $n$
exists, and \[
n=\max\{s,\operatorname{glo}\text{-}\dell[T,U]\}.
\]
 If  $m_s$ is
the number of nonprojective classes in $\mathcal{L}_s(T,U)$ represented by objects outside $\mathrm{add}$P, then $n\le s+m_s$.
\end{lemma}

\begin{corollary}[Derived bounds from finite syzygy strata]
\label{prop:finite-stratum}
Under the hypotheses of Lemma~\ref{lem:stabilization}, let $n$ and $m_s$
be as defined there. Then, for every $k\ge1$,
\[
k\text{-}\ddell B
\le k\text{-}\ddell^{\mathrm{XZ}}[T,U]
\le n+1\le s+m_s+1.
\]
In particular,
\[
\Findim B^{\op}\le\ddell^{\mathrm{XZ}}[T,U]\le n+1.
\]
If $\mathcal H_{T,U}$ is wide, then
\[
k\text{-}\ddell B=k\text{-}\ddell^{\mathrm{XZ}}[T,U]
\le n\le s+m_s,
\qquad \Findim B^{\op}\le n.
\]
\end{corollary}
\begin{proof}
By \cite[Theorem 4.3]{CXZ}, the ordinary finite-test invariant satisfies
\[
k\text{-}\dell^{\mathrm{XZ}}[T,U]\le n+1.
\]
Combining this with Theorem~\ref{thm:canonical} gives
\[
k\text{-}\ddell B\le k\text{-}\ddell^{\mathrm{XZ}}[T,U]
\le k\text{-}\dell^{\mathrm{XZ}}[T,U]\le n+1.
\]
The inequality $n\le s+m_s$ follows from Lemma~\ref{lem:stabilization}.
Moreover, the  inequality \cite[Theorem 2.7]{GI} yields
\[
\Findim B^{\op}\le\ddell B\le\ddell^{\mathrm{XZ}}[T,U].
\]
If the heart is wide, all simple $B$-modules belong to $\operatorname{Im}F$,
so the radical correction terms are absent. The ordinary finite-test
bound therefore improves to $n$ by \cite[Theorem 4.3]{CXZ}.
Theorem~\ref{thm:wide}  gives the
derived bounds.
\end{proof}

\begin{corollary}[Finite representation type]\label{cor:finite-heart}
If $\mathcal H_{T,U}$ is representation finite with $m$ relatively nonprojective
indecomposable classes, then
\[
k\text{-}\ddell^{\mathrm{XZ}}[T,U]\le m+1,
\qquad\Findim B^{\op}\le m+1.
\]
For a wide heart these bounds improve to $m$. In particular, the
conclusions hold for every interval when $A$ is representation finite.
\end{corollary}
\begin{proof}
Apply Corollary~\ref{prop:finite-stratum} with $s=0$. A full summand-closed
subcategory of a representation-finite module category is representation finite.
\end{proof}

\subsection{The first stratum for ICE hearts}

A subcategory is ICE-closed when it is closed under images, cokernels
and extensions. An interval is  an ICE interval if its heart is ICE-closed. In this case Enomoto-Sakai
\cite[Corollary 4.19 and Proposition 4.14]{ES} show that there is a wide $\tau$-tilting module $P$ such that
$\mathcal W=\langle P\rangle_{\mathrm{wide}} $ and $\mathcal H_{T,U}=\Fac_{\mathcal W}P$.
The heart is a torsion class in this wide subcategory. Here
$\Sub_{\mathcal W}P$ consists of the objects of $\mathcal W$ admitting
a monomorphism in $\mathcal W$ into an object of $\add P$.

\begin{proposition}\label{prop:ice-first}\cite[Proposition 4.5]{CXZ}
Assume that $\mathcal H_{T,U}$  is ICE-closed, set $P=P_{T,U}$ and let $\mathcal W=\langle P\rangle_{\mathrm{wide}} $ be the canonical wide subcategory.
Then
\[
\mathrm{add}\mathcal{L}_1(T,U)=\mathcal H_{T,U}\cap\Sub_{\mathcal W}P.
\]
\end{proposition}

\begin{corollary}\label{cor:ice-finite}
Suppose that the interval is ICE and that
$\mathcal H_{T,U}\cap\Sub_{\mathcal W}P$ is representation-finite,
with $m_1$ indecomposable isomorphism classes outside $\add P$.
Then, for every $k\ge1$,
\[
k\text{-}\ddell^{\mathrm{XZ}}[T,U]\le m_1+2,
\qquad\max\{\findim B,\Findim B^{\op}\}\le m_1+2.
\]
\end{corollary}
\begin{proof}
Apply Corollary~\ref{prop:finite-stratum} with $s=1$ for the derived
finite-test and opposite big finitistic-dimension bounds. The bound
for $\findim B$ follows from \cite[Corollary 5.2]{CXZ}.
\end{proof}

\subsection{Ext-injective cogenerators of interval hearts}

All injective objects in this subsection are taken with respect to the
inherited exact structure of $\mathcal H_{T,U}$. Thus $J$ is Ext-injective
if $\Ext^1_{\mathcal H_{T,U}}(X,J)=0$ for every $X$ in the heart. A
finite-sum admissible cogenerator, in the terminology of
\cite[Definition 4.13]{GLZ}, is an object into a finite direct sum of
which every heart object admits an inflation. We first show that our
support $\tau$-tilting endpoints supply such an object, with
Ext-injectivity as well.

\begin{proposition}[An Ext-injective cogenerator]\label{prop:heart-cogenerator}
There is a basic object $J\in\mathcal H_{T,U}$ such that
\[
\operatorname{Inj}\mathcal H_{T,U}=\add J.
\]
Every $X\in\mathcal H_{T,U}$ admits a conflation
\[
0\longrightarrow X\longrightarrow J^{\oplus a}
\longrightarrow Y\longrightarrow0
\]
for some integer $a\ge1$ and some $Y\in\mathcal H_{T,U}$. In particular,
the heart has enough injectives and $J$ is a finite-sum admissible cogenerator.
\end{proposition}
\begin{proof}
By \cite[Proposition 1.1]{AIR}, the torsion-free classes $T^\perp$ and
$U^\perp$ are functorially finite. The standard duality $D$ gives
functorially finite torsion classes
\[
D(U^\perp)\subseteq D(T^\perp)
\quad\text{in }\modu A^{\op}.
\]
By \cite[Theorem 2.7]{AIR}, choose basic support $\tau$-tilting
$A^{\op}$-modules $T'$ and $U'$ with
\[
\Fac T'=D(U^\perp),\qquad\Fac U'=D(T^\perp).
\]
Dualizing the torsion pair $(\Fac U,U^\perp)$ yields
$D(\Fac U)=(D(U^\perp))^\perp$. Hence
\[
D(\mathcal H_{T,U})
=D(T^\perp)\cap(D(U^\perp))^\perp
=\Fac U'\cap(T')^\perp
=\mathcal H_{T',U'}.
\]

By \cite[Lemma 4.17 and Corollary 4.19]{ES}, the module
$U'/\tr_{T'}U'$ is an Ext-progenerator of this dual heart. Choose a
basic module $P'$ with the same additive closure and put $J=DP'$.
Duality makes $J$ Ext-injective in $\mathcal H_{T,U}$. For
$X\in\mathcal H_{T,U}$, a projective conflation
\[
0\longrightarrow K\longrightarrow P'_X\longrightarrow DX\longrightarrow0,
\qquad P'_X\in\add P',
\]
dualizes to a conflation $0\to X\to DP'_X\to DK\to0$ in the original
heart. Choose a complement $J'_X\in\add J$ with
$DP'_X\oplus J'_X\cong J^{\oplus a}$ for some $a\ge1$. Adding the
split conflation $0\to0\to J'_X\xrightarrow{1}J'_X\to0$ gives the
conflation asserted in the proposition, with $Y=DK\oplus J'_X$ still
in the heart.

All objects of $\add J$ are injective. Conversely, the asserted
conflation splits when $X$ is injective, so every injective object
belongs to $\add J$.
\end{proof}

Now, we apply Theorem~4.14 in \cite{GLZ} to the heart $\mathcal H_{T,U}$.

\begin{corollary}\label{prop:cogenerator}
For the object $J$ in Proposition~\ref{prop:heart-cogenerator},
\[
\ddell_{[T,U]}J
\le\operatorname{glo}\text{-}\ddell[T,U]
\le\ddell_{[T,U]}J+1.
\]
The following conditions are equivalent:
\begin{enumerate}
\item $\ddell_{[T,U]}J<\infty$;
\item $\operatorname{glo}\text{-}\ddell[T,U]<\infty$;
\item $\ddell_{[T,U]}X<\infty$ for every $X\in\mathcal H_{T,U}$.
\end{enumerate}
\end{corollary}

\begin{remark}\label{rem:cogenerator-scope}
The estimate and equivalences above are direct applications of the
exact-category theory of \cite{GLZ};
Proposition~\ref{prop:heart-cogenerator} verifies their hypothesis for
the present support $\tau$-tilting intervals. They control the
first-order global invariant. The finite-test comparison in
Theorem~\ref{thm:canonical} involving
$\operatorname{glo}\text{-}(k+1)\text{-}\ddell[T,U]+1$ involves
$\operatorname{glo}\text{-}2\text{-}\ddell[T,U]$ at $k=1$, so the
estimate in Corollary~\ref{prop:cogenerator} alone does not yield the
corresponding finite-test bound for a general interval. If the heart
is wide, Theorem~\ref{thm:wide} and the finitistic-dimension comparison
in Theorem~\ref{thm:canonical} do give
\[
\Findim B^{\op}\le\ddell^{\mathrm{XZ}}[T,U]
=\ddell B\le\ddell_{[T,U]}J+1.
\]
\end{remark}

\subsection{Vanishing of second extensions}

\begin{proposition}[Second extensions]\label{prop:second-ext}
Suppose $\Ext_A^2(X,Y)=0$ for every $X,Y\in\mathcal H_{T,U}$.
Then all first relative syzygies are relatively projective and, for
every $k\ge1$,
\[\operatorname{glo}\text{-}k\text{-}\ddell[T,U]\le1,\qquad k\text{-}\ddell^{\mathrm{XZ}}[T,U]
\le k\text{-}\dell^{\mathrm{XZ}}[T,U]\le2.
\]
In particular, $\operatorname{gl.dim}B\le2$. The above conclusions
hold whenever $A$ is hereditary.
\end{proposition}
\begin{proof}
Take a conflation $0\to K\to P_0\to X\to0$ in the heart, where $P_0\in\mathrm{add}P$. For each
$Y\in\mathcal H_{T,U}$, we have the exact sequence
\[
\Ext_A^1(P_0,Y)\longrightarrow\Ext_A^1(K,Y)
\longrightarrow\Ext_A^2(X,Y).
\]
Thus $K$ is Ext-projective in the heart and lies
in $\add P$. Every $X\in\mathcal H_{T,U}$ has relative projective dimension
at most one, so every ordinary $k$-level is at most one; the derived
bound follows. The finite-test definition gives the bound two. By
\cite[Proposition 5.6]{CXZ}, $\operatorname{gl.dim}B\le2$.
\end{proof}

\section{Examples}\label{sec:examples}

\subsection{A semisimple wide interval}

\begin{example}\label{ex:wide-benchmarks}
Let $A=\kk(1\to2)$, $T=S_2$ and $U=A$. The module $T$ is support
$\tau$-tilting. Since $\Hom_A(S_2,X)\cong X_2$,
\[
\mathcal H_{T,U}=\add S_1,\qquad P_{T,U}=S_1,\qquad B_{T,U}\cong\kk.
\]
Indeed, $\tr_{S_2}P_1=S_2$ and $\tr_{S_2}P_2=P_2$. The heart is wide
and semisimple. Every object is relatively projective, so for all $k\ge1$
\[
k\text{-}\ddell^{\mathrm{XZ}}[T,U]
=\operatorname{glo}\text{-}k\text{-}\ddell[T,U]=0.
\]
The corresponding ordinary levels are also zero.
\end{example}

\subsection{The radical correction term is necessary}

\begin{example}\label{ex:radical-boundary}
Let $A=\kk(1\to2)$, $T=0$ and $U=P_1\oplus S_1$.
The module $U$ is tilting and
\[
\mathcal H_{0,U}=\Fac U=\add(P_1\oplus S_1),\qquad P_{0,U}=U.
\]
The heart is ICE-closed, but not wide: the kernel $S_2$
of $P_1\onto S_1$ is outside the heart. Every intrinsic level is zero.
The algebra $B=\End_A(U)$ is hereditary of type $A_2$.
Write $Q_1=F(P_1)$ and $Q_2=F(S_1)$. Then $Q_1$ is simple projective
and $\rad Q_2\cong Q_1$. Only $\topm Q_1$ is in the image; the radical
lift for the other top is $N_2=P_1$. Therefore, for $k\ge1$,
\[
k\text{-}\ddell^{\mathrm{XZ}}[0,U]
=k\text{-}\dell^{\mathrm{XZ}}[0,U]
=k\text{-}\ddell B=1.
\]
\end{example}

\subsection{A non-ICE interval over a connected algebra attaining the bound}

\begin{example}\label{ex:sharp}
Let $A=\kk(1\to2)$, $U=A$ and $T=S_1$. Then $T$ is support
$\tau$-tilting over the quotient at vertex $1$, $\Hom_A(S_1,A)=0$, and
\[
\mathcal H_{T,U}=S_1^\perp=\operatorname{proj}A
=\add(P_1\oplus S_2),\qquad P=A,\quad B\cong A.
\]
The heart is split exact, so $\operatorname{glo}\text{-}k\text{-}\ddell[T,A]=0$
for every $k\ge1$. The simple $S_2$ is in the image, whereas $S_1$ is not,
and its radical lift is $N_1=S_2$. Hence
$k\text{-}\ddell^{\mathrm{XZ}}[T,A]=1$.
The hereditary algebra has $\Findim A^{\op}=1$, so the finitistic-dimension
bound in Theorem~\ref{thm:canonical} gives $k\text{-}\ddell A\ge1$;
the length-one projective resolutions give the reverse bound. Thus
\[
\Findim B^{\op}=\ddell B=\ddell^{\mathrm{XZ}}[T,A]
=1=\operatorname{glo}\text{-}2\text{-}\ddell[T,A]+1.
\]
The inclusion $S_2\into P_1$ has cokernel $S_1$ outside the heart.
This connected non-ICE example shows that the additive constant one in
the global upper bound of Theorem~\ref{thm:canonical} cannot be removed
in general.
\end{example}
\medskip
\noindent\textbf{Disclosure of computational assistance.}\quad
ChatGPT was used during revision to assist with
literature searches, language editing, and manuscript preparation.

\medskip
\noindent\textbf{Acknowledgements.}\quad
Weiqing Cao is supported by the
National Natural Science Foundation of China (Grant No.~12601071).

\medskip
\noindent\textbf{Data availability.}\quad
Data sharing is not applicable to this article, as no datasets were
generated or analysed during the current study.

\end{document}